\documentclass[12pt]{amsart}
\usepackage[margin=1.35in]{geometry}
\usepackage{amscd,amsmath,amsxtra,amsthm,amssymb,stmaryrd,xr,mathrsfs,mathtools,enumerate,commath, comment, mathtools}
\usepackage{stmaryrd}
\usepackage{multirow}
\usepackage{xcolor}
\usepackage{commath}
\usepackage{comment}
\usepackage{tikz-cd}
\usepackage{longtable} 
\usepackage{pdflscape} 
\usepackage{booktabs}
\usepackage{hyperref}
\definecolor{vegasgold}{rgb}{0.77, 0.7, 0.35}
\definecolor{darkgoldenrod}{rgb}{0.72, 0.53, 0.04}
\definecolor{gold(metallic)}{rgb}{0.83, 0.69, 0.22}
\hypersetup{
 colorlinks=true,
 linkcolor=darkgoldenrod,
 filecolor=brown,      
 urlcolor=gold(metallic),
 citecolor=darkgoldenrod,
 }
\newtheorem{lthm}{Theorem}

\usepackage[all,cmtip]{xy}

\DeclareFontFamily{U}{wncy}{}
\DeclareFontShape{U}{wncy}{m}{n}{<->wncyr10}{}
\DeclareSymbolFont{mcy}{U}{wncy}{m}{n}
\DeclareMathSymbol{\Sh}{\mathord}{mcy}{"58}
\usepackage[T2A,T1]{fontenc}
\usepackage[OT2,T1]{fontenc}

\newtheorem{theorem}{Theorem}[section]
\newtheorem{lemma}[theorem]{Lemma}

\newtheorem*{theorem*}{Theorem}
\newtheorem*{ass*}{Assumption}
\newtheorem{definition}[theorem]{Definition}
\newtheorem{heuristic}[theorem]{Heuristic}

\newtheorem{remark}[theorem]{Remark}

\newtheorem{conjecture}[theorem]{Conjecture}
\newtheorem{proposition}[theorem]{Proposition}

\newcommand{\cH}{\mathcal{H}}

\newcommand{\Z}{\mathbb{Z}}
\newcommand{\Q}{\mathbb{Q}}
\newcommand{\F}{\mathbb{F}}

\newcommand{\cC}{\mathcal{C}}

\newcommand{\N}{\mathbb{N}}

\newcommand{\cyc}{\mathrm{cyc}}

\newcommand{\Prob}{\mathbb{P}}

\newcommand{\op}[1]{\operatorname{#1}}

 \DeclareMathSymbol{\sha}{\mathord}{mcy}{"58}
 \makeatletter
\newcommand{\mylabel}[2]{#2\def\@currentlabel{#2}\label{#1}}
\makeatother

\numberwithin{equation}{section}

\begin{document}

\title[]{A Heuristic approach to the Iwasawa theory of elliptic curves}
\author[K.~M\"uller]{Katharina M\"uller}
\address[Müller]{Institut für Theoretische Informatik, Mathematik und Operations Research, Universität der Bundeswehr München, Werner-Heisenberg-Weg 39, 85577 Neubiberg, Germany}
\email{katharina.mueller@unibw.de}

\author[A.~Ray]{Anwesh Ray}
\address[Ray]{Chennai Mathematical Institute, H1, SIPCOT IT Park, Kelambakkam, Siruseri, Tamil Nadu 603103, India}
\email{anwesh@cmi.ac.in}

\keywords{}
\subjclass[2020]{}

\maketitle

\begin{abstract}
Let $E_{/\mathbb{Q}}$ be an elliptic curve and $p$ an odd prime such that $E$ has good ordinary reduction at $p$ and the Galois representation on $E[p]$ is irreducible. Then Greenberg's $\mu=0$ conjecture predicts that the Selmer group of $E$ over the cyclotomic $\mathbb{Z}_p$-extension of $\Q$ is cofinitely generated as a $\mathbb{Z}_p$-module. In this article we study this conjecture from a statistical perspective. We extend the heuristics of Poonen and Rains to obtain further evidence for Greenberg's conjecture. The key idea is that the vanishing of the $\mu$-invariant can be detected by the intersection $M_1\cap M_2$ of two Iwasawa modules $M_1, M_2$ with additional properties in a given inner product space. The heuristic is based on showing that there is a probability measure on the space of pairs $(M_1, M_2)$, with respect to which the event that $M_1\cap M_2$ is finite happens with probability $1$.
\end{abstract}

\section{Introduction}
\subsection{Background and motivation} Let \( p \) be a prime and \( K \) a number field. The \emph{cyclotomic \( \mathbb{Z}_p \)-extension} \( K_{\operatorname{cyc}} \) is the unique \( \mathbb{Z}_p \)-extension of \( K \) contained in \( K(\mu_{p^\infty}) \). Let \( K_n/K \) be the subextension with \( [K_n:K] = p^n \) and \( h_p(K_n) \) the \( p \)-part of its class number, so \( h_p(K_n) = p^{e_n} \). Iwasawa proved that for large \( n \),  
\[
e_n = p^n \mu_p(K) + n \lambda_p(K) + \nu_p(K),
\]
where \( \mu_p(K), \lambda_p(K) \in \mathbb{Z}_{\geq 0} \) and \( \nu_p(K) \in \mathbb{Z} \). He conjectured \( \mu_p(K) = 0 \) for all \( K \) and \( p \), later proven by Ferrero and Washington for abelian extensions of \( \mathbb{Q} \).

Elliptic curves, fundamental in arithmetic geometry, are smooth projective genus-one curves with a distinguished point. For a number field \( K \), the group of \( K \)-rational points \( E(K) \) is finitely generated by the Mordell--Weil theorem, decomposing into a torsion subgroup and a free part of rank \( r \), the Mordell--Weil rank. This rank is central to the study of elliptic curves and relates to the arithmetic of Selmer groups \( \operatorname{Sel}_p(E/K) \), whose dimension gives an upper bound on \( r \). In Iwasawa theory, elliptic curves are studied over \( \mathbb{Z}_p \)-extensions. For \( E/\mathbb{Q} \) and a prime \( p \) where \( E \) has good ordinary reduction, the \( p \)-primary Selmer group \( \operatorname{Sel}_{p^\infty}(E/\mathbb{Q}_{\operatorname{cyc}}) \) is conjectured to be cofinitely generated and torsion over \( \mathbb{Z}_p[[T]] \).
Kato \cite{kato} proved that \( \operatorname{Sel}_{p^\infty}(E/\mathbb{Q}_{\operatorname{cyc}}) \) is a cofinitely generated and cotorsion module over the Iwasawa algebra. The $\mu$-invariant of the Selmer group is $0$ if it is cofinitely generated as a $\Z_p$-module. Let $E[p]$ be the $p$-torsion subgroup of $E(\bar{\Q})$ and 
\[\rho_{E, p}: \op{Gal}(\bar{\Q}/\Q)\rightarrow \op{Aut}(E[p])\xrightarrow{\sim} \op{GL}_2(\F_p)\] be the Galois representation on $E[p]$, also called the \emph{residual representation}. The following conjecture of Greenberg will be the primary focus of this article. 
\begin{conjecture}
    Let $p$ be an odd prime number and $E$ be an elliptic curve over $\Q$. Assume that $E$ has good ordinary reduction at $p$ and that $\rho_{E,p}$ is irreducible. Then, the $p$-primary Selmer group $\op{Sel}_{p^\infty}(E/\Q_{\op{cyc}})$ is a cofinitely generated $\Z_p$-module. 
\end{conjecture}
\par A theorem of Duke \cite{Dukeexceptional} states that a density $1$ set of such elliptic curves have the property that $\rho_{E, p}$ is surjective. Greenberg's conjecture thus has the following conjectural implication for elliptic curves on average.

\begin{conjecture}[$\mu=0$ on average]\label{mu=0 on average}
    Let $p$ be a fixed odd prime number. Then, for most elliptic curves $E_{/\Q}$ with good ordinary reduction at $p$, the Selmer group $\op{Sel}_{p^\infty}(E/\Q_{\op{cyc}})$ is cofinitely generated over $\Z_p$.
\end{conjecture}
\par Heuristic reasoning has been profoundly influential in the field of arithmetic statistics, offering insightful conjectures and guiding intuition. A cornerstone in this domain is the Cohen-Lenstra heuristics, introduced by Cohen and Lenstra in \cite{CLheuristics}, which propose a probabilistic model for the distribution of class groups of quadratic number fields. Passing from number fields $K$ and their $p$-class groups $\op{Cl}_p(K)$ to Selmer groups of elliptic curves $E$ over $\Q$, Heath--Brown \cite{HeathBrown1, HeathBrown2}, Swinnerton--Dyer \cite{Swinnerton-Dyer} and Kane \cite{Kane} obtained the distribution for the $\F_2$-dimension of the $2$-Selmer group of $E$. Poonen and Rains \cite{poonen} studied heuristics for the distribution of $p$-Selmer group of an elliptic curve over $\Q$ to the intersection of two maximal isotropic subspaces in an inner product space. Inspired by such developments, we provide further evidence for Conjecture \ref{mu=0 on average} via a heuristic approach. Assuming that such intersections are suitably random, the average size of the Selmer groups in question can thus be predicted. This heuristic leads to the following predictions, cf. \cite[Conjecture 1.1 and 1.2]{poonen}:
\begin{enumerate}
    \item $\# \op{Sel}_p(E/\Q)$ has average size $(p+1)$ for any prime $p$.
    \item Asymptotically, $1/2$ of elliptic curves over $\Q$ have rank $0$ and $1/2$ have rank $1$.
    \item As $E$ varies over all elliptic curves over $\Q$, 
    \[\op{Prob}\left(\dim_{\F_p} \op{Sel}_p(E/\Q)=d\right)=\prod_{j\geq 0} \left(1+\frac{1}{p^j}\right)\times \prod_{j=1}^d \left(\frac{p}{p^j-1}\right). \]
\end{enumerate}
When $p\leq 5$, part (1) of the above conjecture was proven by Bhargava and Shankar \cite{BSAnnals1, BSAnnals2, BS5Selmer}. Moreover, they show that there is a positive density of elliptic curves $E_{/\Q}$ with Mordell--Weil rank $0$. Bhargava and Skinner \cite{BhargavaSkinner} prove that a positive density of elliptic curves $E_{/\Q}$ have rank $1$.
In this article, we explore the $\mu = 0$ conjecture, particularly in an "average" sense, by extending the heuristic of Poonen and Rains. The $\mu$-invariant of an elliptic curve $ E/\mathbb{Q} $ is $0$ if and only if the \emph{residual Greenberg Selmer group} $ \operatorname{Sel}^{\operatorname{Gr}}(E[p]/\mathbb{Q}_{\operatorname{cyc}}) $ (cf. \eqref{residual greenberg selmer group}) is finite. Let $ \Lambda $ denote the \textit{Iwasawa algebra} of $ \Gamma := \operatorname{Gal}(\mathbb{Q}_{\operatorname{cyc}}/\mathbb{Q}) $, which is a formal power series ring. We consider the quotient algebra $ \Omega := \Lambda/(p) $, which can be identified with the ring of formal power series $ \mathbb{F}_p\llbracket T \rrbracket $. The structure of the Selmer group $ \operatorname{Sel}_{p^\infty}(E/\mathbb{Q}_{\operatorname{cyc}}) $ is closely related to the \emph{fine Selmer group} $ R_{p^\infty}(E/\mathbb{Q}_{\operatorname{cyc}}) $. The fine Selmer group is defined using more stringent local conditions and plays a key role in Iwasawa theory.
\par A conjecture due to Coates and Sujatha (cf. \cite[Conjecture A]{CoatesSujathaMathAnnalen}) predicts that $ R_{p^\infty}(E/\mathbb{Q}_{\operatorname{cyc}}) $ is cofinitely generated as a $ \mathbb{Z}_p $-module for all elliptic curves over $ \mathbb{Q} $ with good reduction at $ p $. This conjecture is seen to follow as a consequence of Iwasawa's classical $\mu=0$ conjecture (see \cite[Corollary 3.5]{CoatesSujathaMathAnnalen}). The vanishing of the $\mu$-invariant of $ R_{p^\infty}(E/\mathbb{Q}_{\operatorname{cyc}}) $ is captured by the residual fine Selmer group $ R(E[p]/\mathbb{Q}_{\operatorname{cyc}}) $. This is a subgroup of the residual Greenberg Selmer group, which is finite if and only if the $ \mu $-invariant of $ R_{p^\infty}(E/\mathbb{Q}_{\operatorname{cyc}}) $ is $0$ (cf. Proposition \ref{residual fine selmer finite implies conjecture A} for further details).
\subsection{A heuristic for the vanishing of the $\mu$-invariant on average}
\par We describe the steps leading up to our heuristic.
\begin{description}
    \item[Step 1] Let $E$ be an elliptic curve over $\Q$ with good ordinary reduction at $p$. Define $\Pi_E$ to be the quotient of residual Selmer groups $ \Pi_E := \frac{\operatorname{Sel}^{\operatorname{Gr}}(E[p]/\mathbb{Q}_{\operatorname{cyc}})}{R(E[p]/\mathbb{Q}_{\operatorname{cyc}})} $. If $\Pi_E$ is finite, then the Conjecture A of Coates and Sujatha implies that the $\mu$-invariant of  $\op{Sel}_{p^\infty}(E/\Q_{\op{cyc}})$ vanishes.
    \item[Step 2] Let $\Sigma$ be the set of primes $\ell$ such that either $\ell\neq p$ is a prime at which $E$ has bad reduction, or, $\ell=p$. Then we set 
\[V(E):=\bigoplus_{\ell\in \Sigma} \left(\bigoplus_{v|\ell} H^1(\Q_{\op{cyc}, v}, E[p]) \right),\] see \eqref{defn of V(E)}. The $\Omega$-module $V(E)$ comes equipped with a pairing 
\[(,): V(E)\times V(E) \rightarrow \F_p,\] that is symmetric, non-degenerate and bilinear with respect to the action of $\Omega$.
\item[Step 3] The next step is to show that $ \Pi_E $ is isomorphic to the intersection of two $ \Omega $-modules $ V_1(E) $ and $ V_2(E) $, both of $\Omega$ corank $1$ (see Proposition \ref{prop}). The finiteness of $ \Pi_E $ can be interpreted as a question about the finiteness of the intersection of two corank-1 $ \Omega $-modules. Conjecture A of Coates and Sujatha, and the finiteness of $V_1(E)\cap V_2(E)$ implies that the $\mu$-invariant of the Selmer group of $E$ vanishes, see Lemma \ref{boring lemma 1}.
\item[Step 4] By considering projections of the modules $ V_1(E) $ and $ V_2(E) $ onto certain summands of $ V(E) $, we reduce the problem to studying the finiteness of intersections of two submodules $M_1$ and $M_2$ of $(\Omega^\vee)^2$ for which $M_i\simeq \Omega^\vee$. We refer to Proposition \ref{Propn on M1 and M2} for further details.
\end{description}
 Our heuristic is based on the idea that as $E$ is allowed to vary over elliptic curves with good ordinary reduction at $p$, the pairs $(M_1, M_2)$ are equidistributed with respect to a natural probability measure on all such pairs. Our calculations in section \ref{s 4} show that the event of their intersection being finite occurs with probability 1, see Theorem \ref{thm:probability}. This heuristic argument supports Conjecture \ref{mu=0 on average}, which states that on average, the $ \mu $-invariant of elliptic curves over $ \mathbb{Q} $ which good ordinary reduction at $p$ is $0$. 
We now make the heuristic precise. Let \( M_1 \) and \( M_2 \) be \(\Omega\)-submodules of corank \(1\) contained in \((\Omega^\vee)^2\). For each \( i \), let \( N_i \subset \Omega^2 \) be the submodule such that \( M_i^\vee = \Omega^2 / N_i \). Given a natural number \( n \geq 1 \), define \( \Omega_n := \Omega / (T^n) \). A submodule \( N \subset \Omega^2 \) (resp. \( N \subset \Omega_n^2 \)) is said to be \emph{maximal} if it is not contained in \( T\Omega^2 \) (resp. $T\Omega_n^2$). Any cyclic submodule of $\Omega^2$ is contained in a unique maximal submodule. Let \( \mathcal{M} \) (resp. \( \mathcal{M}_n \)) denote the space of pairs of maximal cyclic submodules \( (N_1, N_2) \) contained in \( \Omega^2 \times \Omega^2 \) (resp. \( \Omega_n^2 \times \Omega_n^2 \)). The spaces \( \mathcal{M}_n \) form an inverse system under the natural projection maps \( \mathcal{M}_n \to \mathcal{M}_m \) for \( n \geq m \), allowing us to identify \( \mathcal{M} \) with the inverse limit \( \varprojlim_n \mathcal{M}_n \). We equip \( \mathcal{M} \) with a probability measure \( \mathbb{P} \), defined as the inverse limit of the uniform probability measures \( \mathbb{P}_n \) on each \( \mathcal{M}_n \).  

Now, let \( E \) be an elliptic curve defined over \( \mathbb{Q} \). There exists a unique global minimal Weierstrass model of the form  
\[
E_{A,B} : y^2 = x^3 + A x + B,
\]
where \( (A, B) \) is a pair of integers satisfying the discriminant condition \( \Delta_{A,B} = 4A^3 + 27B^2 \neq 0 \) along with the minimality conditions \( \ell^4 \nmid A \) and \( \ell^6 \nmid B \) for every prime \( \ell \). The naive height of \( E_{A,B} \) is given by  
\[
\operatorname{Ht}(E_{A,B}) := \max \{ |A|^3, |B|^2 \}.
\]
For a positive real number \( x \), define \( \mathcal{C}(x) \) to be the set of all minimal Weierstrass models \( E_{A,B} \) satisfying \( \operatorname{Ht}(E_{A,B}) \leq x \).  

Fix an odd prime \( p \) and define \( \mathcal{T}_p(x) \) as the subset of \( \mathcal{C}(x) \) consisting of curves with good ordinary reduction at \( p \) such that \( V_1(E) \cap V_2(E) \) is infinite. Similarly, define \( \mathcal{S}_p(x) \) as the subset of \( \mathcal{C}(x) \) consisting of curves with good ordinary reduction at \( p \) for which the \( \mu \)-invariant of \( \operatorname{Sel}_{p^\infty}(E/\mathbb{Q}_{\mathrm{cyc}}) \) is nonzero.
\begin{heuristic}\label{heuristic main}
    With respect to notation above,
    \[\lim_{x\rightarrow\infty} \frac{\# \mathcal{T}_p(x)}{\# \mathcal{C}(x)}\leq \mathbb{P}\left(\{(N_1, N_2)\in \mathcal{M}\mid N_1\cap N_2\text{ is infinite}\}\right).\]
\end{heuristic}

We now state the main result of the article.
\begin{lthm}\label{thm A}
    Assume the Heuristic \ref{heuristic main} and Conjecture A of Coates and Sujatha \cite{CoatesSujathaMathAnnalen} for the vanishing of the $\mu$-invariant for the fine Selmer group of any elliptic curve over $\Q$. Then, we have that 
    \[\lim_{x\rightarrow\infty} \frac{\# \mathcal{S}_p(x)}{\# \mathcal{C}(x)}=0.\]
\end{lthm}
 
\par Our results suggest a broader validity of the conjecture on the average vanishing of the $\mu$-invariant, motivating further investigation in more generalized settings. With the growing interest in heuristic models in arithmetic statistics, it is only natural that they will further enrich the interplay between arithmetic statistics and Iwasawa theory. Such  investigations will potentially remain a fertile ground for future research.
\subsection*{Declarations}
\begin{description}
    \item[Consent for publication] The authors give their consent for the publication of the manuscript under review.
     \item[Conflict of interest] Not applicable, there is no conflict of interest to report.
    \item[Availability of data and materials] No data was generated or analyzed in obtaining the results in this article.
    \item[Funding] There are no funding sources to report.
\end{description}

\section{Iwasawa theory of Selmer groups}
\par In this section, we discuss the Iwasawa theory of elliptic curves. For a more comprehensive treatment of the subject, we refer to \cite{CoatesSujathaGCEC,GreenbergIwasawath}.
\subsection{Selmer groups associated to elliptic curves}
\par Let $E$ be an elliptic curve over $\Q$ and $p$ be an odd prime number. Let $K$ be a number field and set $\op{G}_K$ to denote the absolute Galois group $\op{Gal}(\bar{K}/K)$ and let $\Omega_K$ be the finite primes of $K$. For each prime $v\in \Omega_K$, choose an embedding $\iota_v: \bar{K}\hookrightarrow \bar{K}_v$. Setting $\op{G}_{K_v}:=\op{Gal}(\bar{K}_v/K_v)$, find that $\iota_v$ induces an inclusion at the level of Galois groups $\op{G}_{K_v}\hookrightarrow \op{G}_K$. For $F\in \{K, K_v\}$, set \[H^i(F, \cdot):=H^i(\op{G}_F, \cdot)\] and consider the Kummer sequence
\[0\rightarrow \frac{E(F)}{p^n E(F)}\xrightarrow{\kappa_F} H^1(F, E[p^n])\xrightarrow{\theta_F} H^1(F, E)[p^n]\rightarrow 0. \]
The $p^n$ Selmer group $\op{Sel}_{p^n} (E/K)$ consists of all classes $f\in H^1(K, E[p^n])$ such that $\op{res}_v(f)\in \op{Image}(\kappa_{K_v})$ for all primes $v\in \Omega_K$. The $p^\infty$ Selmer group is defined to be the direct limit 
\[\op{Sel}_{p^\infty} (E/K):=\lim_{n\rightarrow \infty} \op{Sel}_{p^n} (E/K).\]
Recall that the Selmer group defined above fits into a natural short exact sequence 
\begin{equation}\label{ses selmer}0\rightarrow E(K)\otimes \Q_p/\Z_p\rightarrow \op{Sel}_{p^\infty}(E/K)\rightarrow \Sh(E/K)[p^\infty]\rightarrow 0,\end{equation}
where $\Sh(E/K)$ denotes the Tate-Shafarevich group of $E$ over $K$. The Tate-Shafarevich group $\Sh(E/K)$ is conjectured to be finite. If $\Sh(E/K)[p^\infty]$ is finite, then, 
\[\op{corank} \op{Sel}_{p^\infty} (E/K)=\op{rank} E(K). \]
\par We let $\Q(\mu_{p^\infty})$ denote the extension of $\Q$ that is generated by the $p$-primary roots of unity and set $\Q_{\op{cyc}}\subset \Q(\mu_{p^\infty})$ to denote the cyclotomic $\Z_p$-extension of $\Q$. Let $\Gamma:=\op{Gal}(\Q_{\op{cyc}}/\Q)$ and choose a topological generator $\gamma\in \Gamma$. Let $\Q_n\subset \Q_{\op{cyc}}$ denote the \emph{n-th layer}, i.e., the subfield for which $[\Q_n:\Q]=p^n$. We identify the Galois group $\op{Gal}(\Q_n/\Q)$ with $\Gamma_n:=\Gamma/\Gamma^{p^n}$. We analyze the structure of Selmer groups over the cyclotomic $\Z_p$-extension of $\Q$, defined as follows
\[\op{Sel}_{p^\infty} (E/\Q_{\op{cyc}}):=\varinjlim_n \op{Sel}_{p^\infty} (E/\Q_{n}).\]
\par The Iwasawa algebra $\Lambda$ is the completed group ring 
\[\Lambda:=\varprojlim_n \Z_p[\Gamma_n].\] Let $\gamma$ be a topological generator of $\Gamma$ and set $T:=(\gamma-1)$. Then $\Lambda$ may be identified with the formal power series ring $\Z_p\llbracket T\rrbracket$. Let $M$ and $M'$ be finitely generated and torsion $\Lambda$-modules, then $M$ is pseudo-isomorphic to $M'$ if there is a $\Lambda$-module map $f: M\rightarrow M'$ whose kernel and cokernel are finite. In this context, the structure theorem \cite[Ch. 13]{washington} asserts that any finitely generated and torsion $\Lambda$-module $M$ is pseudo-isomorphic to $M'$, which is of the form
\[M'=\left(\bigoplus_{i=1}^s \Lambda/(p^{m_i})\right)\oplus \left(\bigoplus_{j=1}^t \Lambda/(f_j)\right).\] Here, $m_i$ are positive integers and $f_j\in \Z_p[T]$ are distinguished polynomials, i.e., monic polynomials whose non-leading coefficients are divisible by $p$. Then the $\mu$-invariant of $M$ is defined to be $\mu(M):=\sum_{i=1}^s m_i$, with the understanding that $\mu(M)=0$ when $s=0$. The $\lambda$-invariant on the other hand is the sum $\lambda(M):=\sum_{i=1}^t \op{deg} f_j$ and taken to be $0$ if $t=0$. The Iwasawa invariants are well defined, i.e., independent of the choice of module $M'$ and its decomposition into cyclic $\Lambda$-modules. It is easy to see that $\mu(M)=0$ if and only if $M$ is finitely generated as a $\Z_p$-module. Furthermore, the $\lambda$-invariant is given by $\lambda(M)=\op{rank}_{\Z_p}(M)$. 

\par A discrete, $p$-primary $\Lambda$-module $M$ has a Pontryagin dual given by \[M^\vee:=\op{Hom}_{\Z_p}\left(M, \Q_p/\Z_p\right).\] We say that $M$ is cofinitely generated (resp. cotorsion) as a $\Lambda$-module if $M^\vee$ is finitely generated (resp. torsion). The $p$-primary Selmer group $\op{Sel}_{p^\infty}(E/\Q_{\op{cyc}})$ is a $p$-primary discrete module over $\Lambda$. Moreover, if $E$ has good ordinary reduction at $p$, the Selmer group $\op{Sel}_{p^\infty}(E/\Q_{\op{cyc}})$ is a cofinitely generated and cotorsion as a $\Lambda$-module. This follows from results of Kato \cite{kato}. Throughout the rest of this article, we shall assume that $E$ is an elliptic curve over $\Q$ with good ordinary reduction at $p$. The $\mu$ and $\lambda$-invariants of the dual Selmer group $\op{Sel}_{p^\infty}(E/\Q_{\op{cyc}})^\vee$ are well defied and denoted by $\mu_p(E)$ and $\lambda_p(E)$ respectively. 

\par The Galois representation on $E[p]$ is denoted by 
\[\rho_{E,p}: \op{G}_{\Q}\rightarrow \op{GL}_2(\F_p).\] The prime $p$ is exceptional for $E$ if $\rho_{E,p}$ is not surjective. Serre's open image theorem implies that if $E$ is an elliptic curve over $\Q$ which is not CM, then, there are only finitely many exceptional primes. 
\begin{conjecture}[Greenberg]
    Let $E_{/\Q}$ be an elliptic curve for which $\rho_{E,p}$ is irreducible. Then, $\mu_p(E)=0$. 
\end{conjecture}

\par Let $p$ be a fixed prime number. It is natural to ask how often $\rho_{E,p}$ is irreducible? Let us make this notion precise. Given an elliptic curve $E_{/\Q}$ there is a unique global minimal Weierstrass model
\[E_{A,B}:y^2=x^3+Ax+B,\]
where $(A,B)$ is a pair of integer for which $\Delta_{A,B}=4A^3+27B^2\neq 0$, and $\ell^4\nmid A$ or $\ell^6\nmid B$ for all primes $\ell$. The naive height of $E_{A,B}$ is defined as follows
\[\op{Ht}(E_{A,B}):=\op{max}\{|A|^3, |B|^2\}.\]
Let $x>0$ be a positive real number and set to be the set of minimal Weierstrass models $E_{A,B}$ such that $\op{Ht}(E_{A,B})\leq x$. There is a natural identification of $\cC(x)$ with the following set: 
\[
\begin{Bmatrix}
 & \abs{A}\leq \sqrt[3]{X}, \ \abs{B}\leq \sqrt{X}\\
(A,B) \in \Z \times \Z: & \quad 4A^3 +27B^2 \neq 0 \\
& \textrm{for all primes } \ell \textrm{ if } \ell^4 |A, \textrm{then } \ell^6\nmid B 
\end{Bmatrix}.
\]
Given a set $S$ of elliptic curves $E_{/\Q}$, set $S(x):=S\cap \cC(x)$. The density of $S$ is defined as the following limit 
\[\mathfrak{d}(S):=\lim_{x\rightarrow \infty} \frac{\# S(x)}{\# \cC(x)},\] provided it exists. The upper (resp. lower) limit $\overline{\mathfrak{d}}(S)$ (resp. $\underline{\mathfrak{d}}(S)$) is defined by replacing the above limit with $\limsup$ (resp. $\liminf$). We say that $S$ consists of $c\%$ of elliptic curves if $\mathfrak{d}(S)$ exists and equals $c/100$. On the other hand, we say that $S$ consists of at least (resp. at most) $c \% $ of elliptic curves if $\underline{\mathfrak{d}}(S)\geq c/100$ (resp. $\overline{\mathfrak{d}}(S)\leq c/100$).

\par It follows from results of Duke \cite{Dukeexceptional} that $\rho_{E,p}$ is surjective for $100\%$ elliptic curves $E_{/\Q}$ when ordered by naive height. In fact, it is shown that $100\%$ of elliptic curves have no exceptional primes. Thus, Greenberg's conjecture has the following consequence (for a fixed odd prime number $p$).

\begin{conjecture}[$\mu=0$ on average -- precise version]\label{ detailed conjecture}
    Fix an odd prime number $p$. Let $S$ be the set of elliptic curves with good ordinary reduction at $p$ for which $\mu_p(E)>0$. Then, we have that $\mathfrak{d}(S)=0$. 
\end{conjecture}

\par Let $E$ be an elliptic curve over $\Q$ with good ordinary reduction at $p$. We recall a variant of the Selmer group which was introduced by Greenberg \cite{Greenberpadic}. Let $\Sigma$ be the set of primes $\ell$ such that either $E$ has bad reduction at $\ell$ or $\ell=p$. For each prime $\ell\in \Sigma$, there is a local condition $\cH_\ell(E/\Q_{\op{cyc}})$ called the Greenberg local condition. These are defined as follows. For $\ell\in \Sigma\setminus \{p\}$, we set 
\[\cH_\ell(E/\Q_{\op{cyc}}):=\bigoplus_{v|\ell} H^1(\Q_{\op{cyc}, v}, E[p^\infty]),\] where $v$ ranges over the primes of $\Q_{\op{cyc}}$ that lie above $\ell$. To define the condition at $p$, let $\eta_p$ denote the unique prime of $\Q_{\op{cyc}}$ that lies above $p$. Note that $E$ has good ordinary reduction at $p$, and thus there is a natural decomposition of $p$-primary $\op{G}_{p}$-modules
\[0\rightarrow \widehat{E}[p^\infty]\rightarrow E[p^\infty]\rightarrow \widetilde{E}[p^\infty]\rightarrow 0,\] where $\widehat{E}$ is the formal group of $E$ and $\widetilde{E}$ is the reduction of the Neron model of $E$ at $p$. Let $\op{I}_{p}$ denote the inertia subgroup of $\op{G}_p$ and note that $\widetilde{E}$ is unramified as a $\op{G}_p$-module, i.e., $\op{I}_p$ acts trivially on it. Both $\widehat{E}[p^\infty]$ and $\widetilde{E}[p^\infty]$ are isomorphic to $\Q_p/\Z_p$. Then at $p$, we set 
\[\cH_p(E/\Q_{\op{cyc}}):= \frac{H^1(\Q_{\op{cyc}, \eta_p}, E[p^\infty])}{\op{ker}\left(H^1(\Q_{\op{cyc}, \eta_p}, E[p^\infty])\longrightarrow H^1(\op{I}_{\eta_p}, \widetilde{E}[p^\infty])\right)}.\]
\begin{definition}
    With respect to notation above, the Greenberg Selmer group is defined as follows
    \[\op{Sel}^{\op{Gr}}(E/\Q_{\op{cyc}}):=\op{ker}\left(H^1(\Q_\Sigma/\Q_{\op{cyc}}, E[p^\infty])\longrightarrow \bigoplus_{\ell\in \Sigma} \cH_\ell(E/\Q_{\op{cyc}})\right)\]
\end{definition}
We note that the Selmer group $\op{Sel}_{p^\infty}(E/\Q_{\op{cyc}})$ and $\op{Sel}_{p^\infty}^{\op{Gr}}(E/\Q_{\op{cyc}})$ coincide \cite[Ch.2, Propositions 2.1 and 2.4]{GreenbergITEC}, and thus we shall simply denote them by $\op{Sel}_{p^\infty}(E/\Q_{\op{cyc}})$. 

\subsection{The fine Selmer group}
\par In this section we introduce the fine Selmer group associated to an elliptic curve and recall its basic properties. Throughout, $E$ will be an elliptic curve over $\Q$ and $p$ will be an odd prime number at which $E$ has good ordinary reduction. Given a rational prime $\ell$, let $J_\ell(E/\Q_{\op{cyc}})$ denote the direct sum 
\[J_\ell(E/\Q_{\op{cyc}}):=\bigoplus_{v|\ell} H^1(\Q_{\op{cyc}, v}, E[p^\infty]),\] where the sum is over the primes of $\Q_{\op{cyc}}$ that lie above $\ell$. Let $\Sigma$ be a finite set of rational primes that contain $p$ and the primes at which $E$ has bad reduction. Then, the fine Selmer group of $E$ over $\Q_{\op{cyc}}$ is defined as follows:
\begin{equation}\label{def of fine selmer}R_{p^\infty}(E/\Q_{\op{cyc}}):=\op{ker}\left( H^1(\Q_{\Sigma}/\Q_{\op{cyc}}, E[p^\infty])\longrightarrow \bigoplus_{\ell\in \Sigma} J_\ell(E/\Q_{\op{cyc}})\right). \end{equation}
The definition is in fact independent of the choice of $\Sigma$, cf. \cite[section 3]{sujathawitte}. The fine Selmer group is a cofinitely generated $\Lambda$-module. When $E$ has good ordinary reduction at $p$, then it is known that $R_{p^\infty}(E/\Q_{\op{cyc}})$ is also cotorsion as a $\Lambda$-module. We recall the conjecture of Coates of Sujatha \cite[Conjecture A]{CoatesSujathaMathAnnalen}, which is the analogue of Iwasawa's $\mu=0$ conjecture.
\begin{conjecture}[Coates and Sujatha]
    Let $E$ be an elliptic curve over $\Q$ and $p$ be an odd prime number at which $E$ has good ordinary reduction. The $\mu$-invariant of $R_{p^\infty}(E/\Q_{\op{cyc}})$ is equal to $0$. Equivalently, $R_{p^\infty}(E/\Q_{\op{cyc}})$ is a cofinitely generated $\Z_p$-module. 
\end{conjecture}

\par The vanishing of the $\mu$-invariant can be detected from the residual Selmer group associated to $E[p]$. Given a rational prime number $\ell$, set
\[J_\ell(E[p]/\Q_{\op{cyc}}):=\bigoplus_{v|\ell} H^1(\Q_{\op{cyc}, v}, E[p]).\]
The residual fine Selmer group is defined as follows 
\[R(E[p]/\Q_{\op{cyc}}):=\op{ker}\left( H^1(\Q_{\Sigma}/\Q_{\op{cyc}}, E[p])\longrightarrow \bigoplus_{\ell\in \Sigma} J_\ell(E[p]/\Q_{\op{cyc}})\right). \]
\begin{proposition}\label{residual fine selmer finite implies conjecture A}
    With respect to notation above, the following are equivalent:
    \begin{enumerate}
        \item $R_{p^\infty}(E/\Q_{\op{cyc}})$ is cofinitely generated as a $\Z_p$-module;
        \item $R(E[p]/\Q_{\op{cyc}})$ is finite.
    \end{enumerate}
\end{proposition}
\begin{proof}
    Note that every rational prime $\ell$ decomposes into finitely many primes in $\Q_{\op{cyc}}$. The result consequently follows (for instance) from \cite[Proposition 3.1]{RayRamanujanJcorank}.
\end{proof}

\begin{proposition}\label{reducible mu=0}
    Suppose the $\rho_{E,p}$ is reducible, then, $R_{p^\infty}(E/\Q_{\op{cyc}})$ is cofinitely generated as a $\Z_p$-module.
\end{proposition}
\begin{proof}
    The result follows from \cite[Proposition 2.4]{arithstatsfineselmer}.
\end{proof}

\begin{theorem}[Coates--Sujatha]
    Let $K$ be the number field $\Q(E[p])$ and assume that the classical Iwasawa $\mu$-invariant $\mu_p(K)=0$. Then it follows that $R_{p^\infty}(E/\Q_{\op{cyc}})$ is cofinitely generated as a $\Z_p$-module.
\end{theorem}
\begin{proof}
    The result is a special case of \cite[Corollary 3.5]{CoatesSujathaMathAnnalen}.
\end{proof}

\section{The residual Selmer group and the vanishing of the $\mu$-invariant}
\label{s 3}
\par Throughout the rest of this article, $E_{/\Q}$ is an elliptic curve and $p$ is an odd prime, such that the following conditions are satisfied: 
\begin{enumerate}
    \item $E$ has good ordinary reduction at $p$, 
    \item $\rho_{E,p}$ is irreducible. In particular, it follows that $E(\Q)[p]=0$.
\end{enumerate}
The ordinary condition can be detected by a congruence condition modulo $p$ for the Weierstrass coefficients of $E$, and is satisfied for a positive density set of elliptic curves $E_{/\Q}$. On the other hand, it follows from Duke's theorem that $\rho_{E,p}$ is surjective for almost all elliptic curves. Thus condition (2) is satisfied for $100\%$ of elliptic curves ordered by height. The vanishing of the $\mu$-invariant can be detected by the structure of the Greenberg Selmer group associated to the residual representation. Throughout, we assume that $p$ is odd and let $\Sigma$ be the set of primes of $\Q$ consisting of primes $p$ and the primes at which $\rho_{E,p}$ is ramified. Given $\ell\in \Sigma\backslash \{p\}$, set \[\cH_\ell(E[p]/\Q_{\op{cyc}}):=\bigoplus_{v|\ell} H^1(\Q_{\op{cyc}, v}, E[p]),\] where $v$ ranges over the primes of $\Q_{\op{cyc}}$ that lie above $\ell$. On the other hand, set \[\cH_p(E[p]/\Q_{\op{cyc}}):=\frac{H^1(\Q_{\op{cyc}, \eta_p}, E[p])}{\op{ker}\left(H^1(\Q_{\op{cyc}, \eta_p}, E[p])\longrightarrow H^1(\op{I}_{\eta_p}, \widetilde{E}[p])\right)}.\]
The residual Selmer group is defined as follows
\begin{equation}\label{residual greenberg selmer group}\op{Sel}^{\op{Gr}}(E[p]/\Q_{\op{cyc}}):=\op{ker}\left(H^1(\Q_{\op{cyc}}, E[p])\longrightarrow \bigoplus_{\ell\in \Sigma} \cH_\ell(E[p]/\Q_{\op{cyc}})\right).\end{equation}
Let $\Omega$ be the mod-$p$ quotient $\Lambda/(p)$, which we identify with the power series ring $\F_p\llbracket T \rrbracket$. We assume throughout that $p$ is odd and $E$ has good ordinary reduction at $p$.

\begin{lemma}
\label{lem:local}
    With respect to notation above, the following assertions hold:
\begin{enumerate}
    \item given a prime $v$ of $\Q_{\op{cyc}}$ such that $v\nmid p$, we have that $H^1(\Q_{\op{cyc}, v}, E[p])$ is finite; 
    \item $H^1(\Q_{\op{cyc}, \eta_p}, E[p])\simeq \Omega^\vee\oplus \Omega^\vee\oplus W$,
    where $W$ is a finite module. 
\end{enumerate}

\end{lemma}
\begin{proof}
We begin by proving part (1). 
We have a natural exact sequence
   \begin{equation}\label{ses cyc}E(\Q_{\op{cyc},v})[p^\infty]/pE(\Q_{\op{cyc},v})[p^\infty]\to H^1(\Q_{\op{cyc},v},E[p])\to H^1(\Q_{\op{cyc},v},E[p^\infty])[p]\to 0. \end{equation}
The module $H^1(\Q_{\op{cyc},v},E[p^\infty])$ cofinitely generated as a $\Z_p$-module (cf. \cite[p.33]{greenberg-vatsal}), and therefore it follows that $H^1(\Q_{\op{cyc},v},E[p^\infty])[p]$ is finite. It is clear that $E(\Q_{\op{cyc},v})[p^\infty]/pE(\Q_{\op{cyc},v})[p^\infty]$ is finite as well. Thus part (1) follows.
  \par In order to prove (2), set $v:=\eta_p$. We find that the first term in \eqref{ses cyc} is clearly finite and the last term has $\Omega$-corank $2$ by \cite[Proposition 1, p.109]{Greenberpadic}. It follows that $H^1(\Q_{\op{cyc},\eta_p},E[p])$ has $\Omega$-corank $2$. As $\Omega$ is a principal ideal domain we obtain a decomposition
   \[H^1(\Q_{\op{cyc},\eta_p},E[p])=\Omega^\vee\oplus \Omega^\vee \oplus W,\]
   where $W$ is a finite submodule.
\end{proof}
\begin{lemma}\label{mu=0 criterion lemma} With respect to notation above, the following assertions hold:
\begin{enumerate}
    \item there is a natural map 
    \[\psi\colon \op{Sel}^{\op{Gr}}(E[p]/\Q_{\op{cyc}})\rightarrow \op{Sel}^{\op{Gr}}(E/\Q_{\op{cyc}})[p]\] with finite kernel and cokernel.
    \item  We have that $\mu_p(E)=0$ if and only if $\op{Sel}^{\op{Gr}}(E[p]/\Q_{\op{cyc}})$ is finite.
\end{enumerate}

\end{lemma}
\begin{proof}
   From the Kummer sequence,
   \[0\to E[p]\to E[p^\infty]\to E[p^\infty]\to 0\] 
   one has a natural commutative diagram
   \begin{equation}\label{commutative diagram selmer}
\begin{tikzcd}[column sep = small, row sep = large]
0\arrow{r} & \op{Sel}^{\op{Gr}}(E[p]/\Q_{\op{cyc}})\arrow{r} \arrow{d}{\psi} & H^1\left(\Q_\Sigma/\Q,E[p]\right) \arrow{r} \arrow{d}{\Psi} & \bigoplus_{\ell\in \Sigma} \cH_\ell(E[p]/\Q_{\op{cyc}}) \arrow{d}{h} &\\
0\arrow{r} & \op{Sel}^{\op{Gr}}(E/\Q_{\op{cyc}})[p] \arrow{r} & H^1\left(\Q_\Sigma/\Q_{\op{cyc}},E[p^\infty]\right)[p] \arrow{r} &\bigoplus_{\ell\in \Sigma} \cH_\ell(E/\Q_{\op{cyc}})[p].
\end{tikzcd}\end{equation}
    Here, the map $h$ is the direct sum of natural maps
    \[h_\ell:\cH_\ell(E[p]/\Q_{\op{cyc}})\rightarrow  \cH_\ell(E/\Q_{\op{cyc}})[p] \]
    over the set $\ell\in \Sigma$. Since it is assumed that $E(\Q)[p]=0$, and $\Q_{\op{cyc}}/\Q$ is a pro-$p$ extension, it follows that $E(\Q_{\op{cyc}})[p]=0$. From the inflation-restriction sequence, we find that $\Psi$ is an isomorphism. Thus, $\psi$ is injective.
   
 \par It remains to show that the cokernel of $\psi$ is finite. Let $\Sigma_0$ be the places of bad reduction and define
    \[\op{Sel}_{\Sigma_0}^{\op{Gr}}(E[p]/\Q_{\op{cyc}}):=\op{ker}\left(H^1(\Q_{\op{cyc}}, E[p])\longrightarrow \bigoplus_{\ell\in \Sigma\setminus \Sigma_0} \cH_\ell(E[p]/\Q_{\op{cyc}})\right)\]
    and
    \[\op{Sel}_{\Sigma_0}^{\op{Gr}}(E/\Q_{\op{cyc}}):=\op{ker}\left(H^1(\Q_{\op{cyc}}, E[p^\infty])\longrightarrow \bigoplus_{\ell\in \Sigma\setminus \Sigma_0} \cH_\ell(E[p^\infty]/\Q_{\op{cyc}})\right).\] By \cite[Proposition 2.4]{greenberg-vatsal} and Lemma \ref{lem:local} (1) the natural maps
    \[\op{Sel}^{\op{Gr}}(E/\Q_{\op{cyc}})[p] \to \op{Sel}_{\Sigma_0}^{\op{Gr}}(E/\Q_{\op{\cyc}})[p]\] and
     \[\op{Sel}^{\op{Gr}}(E[p]/\Q_{\op{cyc}}) \to \op{Sel}_{\Sigma_0}^{\op{Gr}}(E[p]/\Q_{\op{\cyc}})\]
     are injective and have finite cokernel. By \cite[Proposition 2.8]{greenberg-vatsal} we have 
     \[\op{Sel}_{\Sigma_0}^{\op{Gr}}(E[p]/\Q_{\op{cyc}})\cong \op{Sel}^{\op{Gr}}_{\Sigma_0}(E/\Q_{\op{cyc}})[p]. \]
     It follows that $\psi$ has finite cokernel as well. This proves part (1). 

     \par For part (2), first observe that from the structure theory of $\Lambda$-modules, $\mu_p(E)=0$ if and only $\op{Sel}_{p^\infty}(E/\Q_{\op{cyc}})[p]$ is finite. Thus part (2) is a consequence of (1).
\end{proof}
We note here that a cofinitely generated $\Omega$-module is finite if and only if it is cotorsion. Thus, Lemma \ref{mu=0 criterion lemma} asserts that 
\[\mu_p(E)=0\Leftrightarrow \op{corank}_{\Omega}\op{Sel}^{\op{Gr}}(E[p]/\Q_{\op{cyc}})=0.\]

The Weil pairing $\langle,\rangle: E[p]\times E[p]\rightarrow \mu_p$ induces a non-degenerate pairing 
\[[,]_v: H^1(\Q_{\op{cyc}, v}, E[p])\times H^1(\Q_{\op{cyc}, v}, E[p])\rightarrow \F_p.\] Taking the direct sum of primes $v$ of $\Q_{\op{cyc}}$ that lie above $\ell$, one obtains a pairing 
\[(,)_\ell: \left(\bigoplus_{v|\ell} H^1(\Q_{\op{cyc}, v}, E[p])\right)\times \left(\bigoplus_{v|\ell} H^1(\Q_{\op{cyc}, v}, E[p])\right)\rightarrow \F_p\]
defined by 
\[(a,b)_\ell:=\sum_{v|\ell} [a,b]_v.\]
Let $\iota$ denote the involution on $\Gamma$, taking $\gamma$ to $\gamma^{\iota}:=\gamma^{-1}$. Set $V(E)$ to be the $\Omega$-module 
\begin{equation}\label{defn of V(E)}V(E):=\bigoplus_{\ell\in \Sigma} \left(\bigoplus_{v|\ell} H^1(\Q_{\op{cyc}, v}, E[p]) \right),\end{equation}
equipped with the pairing 
\[(,): V(E)\times V(E) \rightarrow \F_p,\] defined by 
\[(a,b):=\sum_{\ell\in \Sigma} (a,b)_\ell.\] This pairing has the property that it is bilinear, non-degenerate and that for $a,b\in V(E)$ and $\tau\in \Omega$, one has that 
\begin{equation}(\tau a, b)=(a, \tau^\iota b),\end{equation} cf. \cite[Proposition 1.5.3]{NSW}.

\begin{definition}
    Let $M$ be a subgroup of $V(E)$ and $M^{\perp}$ be the orthogonal complement of $M$ with respect to $(\cdot, \cdot)_{V(E)}$, defined as follows:
    \[M^{\perp}:=\{m\in V(E)\mid (m, m')=0\text{ for all }m'\in M\}.\]
    Then $M$ is said to be \emph{isotropic} if it is contained in $M^{\perp}$. Moreover, $M$ is \emph{maximal isotropic} if $M=M^{\perp}$.
\end{definition}
We note that when $M$ is an $\Omega$-submodule of $V(E)$, then $M^{\perp}$ is also an $\Omega$-submodule. Consider the natural map of $\Omega$-modules induced by restriction
\[\theta: H^1(\Q_\Sigma/\Q_{\op{cyc}}, E[p])\rightarrow V(E)\] and let $V_1(E)$ be the image of $\theta$. Note that the kernel of $\theta$ is the residual fine Selmer group $R(E[p]/\Q_{\op{cyc}})$. 

\begin{lemma}\label{V1 is maxl iso}
    With respect to notation above, $V_1(E)$ is a maximal isotropic $\Omega$-submodule of $V(E)$ .
\end{lemma}
\begin{proof}
    We write $V(E)$ as a direct limit of $\Omega$-modules $V(E)=\varinjlim_n V^n(E)$,
    where \[V^n(E):=\bigoplus_{\ell\in \Sigma} \left(\bigoplus_{v|\ell} H^1(\Q_{n, v}, E[p]) \right).\]
    Likewise, $V_1(E)$ is a direct limit $\varinjlim_n V_1^n(E)$, where 
    \[V_1^n(E):=\op{image}\{H^1(\Q_\Sigma/\Q_n, E[p])\longrightarrow V^n(E)\}.\]

    The pairing $(,)_{V(E)}$ on $V(E)$ is the direct limit of pairings $(,)_{V^n(E)}$. It follows from this that $V_1(E)$ is isotropic as a submodule of $V(E)$. Each of the submodules $V_1^n(E)$ are maximal isotropic with respect to the pairing $(,)_{V^n(E)}$ by \cite[Theorem 4.14 and remark 4.15]{poonen}. It follows that $V_1(E)$ is maximal isotropic. In greater detail, let $a\in V_1(E)$ and $b\in V(E)$ be such that $(a,b)=0$. Then we write $a$ (resp. $b$) as a limit of $(a_n)$ (resp. $(b_n)$). Here, $a_n\in V_1^n(E)$ and $b_n\in V^n(E)$, and note that 
    \[(a,b)=\lim_n (a_n, b_n).\] It follows that $(a_n, b_n)=0$ for all large enough values of $n$. This implies that $b_n\in V_1^n(E)$ for all large enough values of $n$ and therefore, $b=\lim_n b_n$ belongs to $V_1(E)$. This shows that $V_1(E)$ is maximal. 
\end{proof}
Let \[V_2(E):=\op{ker}\left(H^1(\Q_{\op{cyc}, \eta_p}, E[p])\longrightarrow H^1(\op{I}_{\eta_p}, \widetilde{E}[p])\right),\] viewed as an $\Omega$-submodule of $V(E)$. 
\begin{lemma} \label{V2 corank lemma} 
    As an $\Omega$-module, we have that
    \[V_2(E)\simeq \Omega^\vee \oplus W'\]
    for a finite submodule $W'$
\end{lemma}
\begin{proof}
Let $\phi\colon H^1(\Q_{\op{cyc},\eta_p},E[p])\longrightarrow H^1(\op{I}_{\eta_p}, \widetilde{E}[p])$. We obtain a commutative diagram
\begin{equation}
\begin{tikzcd}[column sep = small, row sep = large]
0\arrow{r} & V_2(E)\arrow{r}{\phi} \arrow{d}{\phi'} & H^1\left(\Q_{\op{cyc},\eta_p},E[p]\right) \arrow{r} \arrow{d}{\iota} & \textup{Im}(\phi) \arrow{d} \arrow{r}&0\\
0\arrow{r} & L_{\eta_p}[p] \arrow{r} & H^1\left(\Q_{\op{cyc},\eta_p},E[p^\infty]\right)[p] \arrow{r} &H^1(I_{\eta_p},\widetilde{E}[p^\infty])&,
\end{tikzcd}\end{equation}
here $I_{\eta_p}$ is the inertia subgroup at $\eta_p$ and $L_{\eta_p}$ is the canonical module making the bottom row exact. 
By the long exact sequence in cohomology, $\iota$ is surjective and has finite kernel. Furthermore, the right vertical map has kernel that is a subgroup of 
\[H^0(I_{\eta_p},\widetilde{E}[p^\infty])/pH^0(I_{\eta_p},\widetilde{E}[p^\infty])=0.\]
By the snake lemma $\phi'$ is surjective and has finite kernel. It therefore suffices to determine the $\Omega$-corank of $L_{\eta_p}[p]$.
By \cite[page 42]{greenberg-vatsal} we know that $L_{\eta_p}[p]\cong (E(\Q_{\op{cyc},\eta_p}) \otimes \Q_p/\Z_p)[p]$.
    By \cite{greenberg-vatsal} the quotient $H^1(\Q_{\op{cyc}\eta_p},E[p^\infty])/E(\Q_{\op{cyc},\eta_p})\otimes \Q_p/\Z_p$ has $\Lambda$-corank $1$. By \cite[Proposition 1]{Greenberpadic} $H^1(\Q_{\op{cyc},\eta_p},E[p^\infty])$ has $\Lambda$-corank $2$. As $E(\Q_{\op{cyc},\eta_p})\otimes \Q_p/\Z_p$ is $\Z_p$-divisible it follows that 
    \[\Omega\textup{-corank}((E(\Q_{\op{cyc},\eta_p})\otimes \Q_p/\Z_p)[p])=\Lambda\textup{-corank}(E(\Q_{\op{cyc},\eta_p})\otimes \Q_p/\Z_p)=2-1=1.\]
    As $\Omega$ is a principal ideal domain the desired claim follows.
\end{proof}
We may write
\[V(E)=\left(\Omega^\vee a \oplus \Omega^\vee b\right) \oplus \left(\bigoplus_{i=1}^n \left( \frac{\Omega}{(T^{m_i})} e_i \oplus \frac{\Omega}{(T^{m_i})} f_i\right)\right), \]
where $\left(\Omega^\vee a \oplus \Omega^\vee b\right)=H^1(\Q_{\op{cyc},\eta_p},E[p])$. 
The $\Omega$-equivariant pairing satisfies 
\[
    (a, b)=(e_i, f_i)=1\text{ and } (b,a)=(f_i, e_i)=1,
\]
and all other pairings vanish. We note that for integers $k_1, k_2\in [0, m_i)$ and \[\delta_{k_1, k_2}=\begin{cases} &1 \text{ if }k_1=k_2;\\
&0 \text{ if }k_1\neq k_2.\\\end{cases}\]
\[(\gamma^{k_1} a, \gamma^{k_2} b)=\delta_{k_1, k_2}(a,b)\text{ and }(\gamma^{k_1} e_i, \gamma^{k_2} f_i)=\delta_{k_1, k_2}(a,b).\]

\begin{lemma}\label{str of maxl iso}
    Let $W$ be a maximal isotropic submodule of $V(E)$. Then, $W$ decomposes as 
    \[W=\Omega^\vee c\oplus W_{\op{tors}},\] where $\Omega^\vee c$ is contained in $\Omega^\vee a\oplus \Omega^\vee b$ and $W_{\op{tors}}$ is a torsion $\Omega$-module.
\end{lemma}
\begin{proof}
    Clearly, $W$ has $\Omega$-corank equal to $1$. Thus, 
    \[W^\vee \cong \Omega \oplus W',\] where $W'$ is a $\Omega$-torsion module. Thus, $W$ decomposes as 
    \[W=\Omega^\vee c'\oplus W_{\op{tors}},\]
where $c'= c+\sum_{i=1}^n c_i e_i+\sum_{j=1}^n d_i f_i$, where $c\in \Omega^\vee a\oplus \Omega^\vee b$ and $c_i, d_i\in \Omega/(T^{m_i})$. Note that $W_{\op{tors}}$ is a maximal isotropic $\Omega$-submodule of 
\[V(E)_{\op{tors}}=\left(\bigoplus_{i=1}^n \left( \frac{\Omega}{(T^{m_i})} e_i \oplus \frac{\Omega}{(T^{m_i})} f_i\right)\right). \] Without loss of generality, 
\[W_{\op{tors}}=\bigoplus_{i=1}^n \frac{\Omega}{(T^{m_i})} e_i.\]
Thus, after subtracting $\sum_{i} c_i e_i$, we can assume that 
\[c'=c+\sum_{j=1}^n d_j f_i.\] On the other hand, $(c', \gamma^k e_i)=0$ for all $k$ and $i$. Since 
\[(c',\gamma^k e_i)=(d_if_i,\gamma^ke_i)\] for all $k$, it follows that $d_i=0$ and thus, $c'=c$. This completes the proof.
\end{proof}

Recall that from Lemma \ref{V1 is maxl iso}, $V_1(E)$ is maximal isotropic and thus in view of Lemma \ref{str of maxl iso}, we have that $V_1(E)=M_1\oplus V_1(E)_{\op{tors}}$, where $M_1\simeq \Omega^\vee$ and is contained in $\Omega^\vee a\oplus \Omega^\vee b$. On the other hand, by Lemma \ref{V2 corank lemma}, $V_2(E)=M_2'\oplus V_2(E)_{\op{tors}}$, where $M_2'\simeq \Omega^\vee$. We let $M_2$ be the projection of $M_2'$ onto the summand $\Omega^\vee a\oplus \Omega^\vee b$.
\begin{lemma}\label{boring lemma 1}
    With respect to notation above, the following are equivalent:
    \begin{enumerate}
        \item the $\mu$-invariant of $\op{Sel}_{p^\infty}(E/\Q_{\op{cyc}})$ is $0$;
        \item the $\mu$-invariant of $R_{p^\infty}(E/\Q_{\op{cyc}})$ is $0$ and $V_1(E)\cap V_2(E)$ is finite.
    \end{enumerate}
\end{lemma}
\begin{proof}
    Note that \[\theta\left(\op{Sel}^{\op{Gr}}(E[p]/\Q_{\op{cyc}})\right)=\op{image}(\theta)\cap V_2(E)=V_1(E)\cap V_2(E).\] Thus, the residual Selmer group $\op{Sel}^{\op{Gr}}(E[p]/\Q_{\op{cyc}})$ is finite if and only if the $\op{ker}\theta$ and $V_1(E)\cap V_2(E)$ are both finite. Recall that the $\mu$-invariant of $\op{Sel}_{p^\infty}(E/\Q_{\op{cyc}})$ is $0$ if and only if $\op{Sel}^{\op{Gr}}(E[p]/\Q_{\op{cyc}})$ is finite. On the other hand, the kernel of $\theta$ is $R(E[p]/\Q_{\op{cyc}})$ which is finite if and only if the $\mu$-invariant of $R_{p^\infty}(E/\Q_{\op{cyc}})$ is $0$.
\end{proof}

As an immediate consequence of the above lemma we obtain
\begin{proposition}\label{Propn on M1 and M2}
\label{prop}
    With respect to notation above, the following assertions hold:
    \begin{enumerate}
        \item Suppose that $M_1\cap M_2$ is finite. Then it follows that $V_1(E)\cap V_2(E)$ is finite.
        \item Suppose that the fine Selmer group $R_{p^\infty}(E/\Q_{\op{cyc}})$ has $\mu$-invariant equal to $0$ and that $M_1\cap M_2$ is finite. Then, the $\mu$-invariant of $\op{Sel}_{p^\infty}(E/\Q_{\op{cyc}})$ is $0$.
    \end{enumerate}
\end{proposition}
\begin{proof}
    Part (1) is clear and left to the reader. Part (2) then follows from Lemma \ref{boring lemma 1}.
\end{proof}

\section{A Heuristic for $\Omega$-modules}\label{s 4}
Recall that $\Omega=\mathbb{F}_p\llbracket T\rrbracket$ and that $\Omega^\vee$ denotes its Pontryagin dual. Let $M_1,M_2$ be the submodules of $(\Omega^\vee)^2$ defined in the previous section. We would like to compute the probability that $M_1\cap M_2$ is finite. Let $N_i\subset \Omega^2$ be the submodule such that
\[M_i^\vee=\Omega^2/N_i.\]
For any natural number $n\ge 1$ we define $\Omega_n\colon =\Omega/(T^n)$.
\begin{definition}
    We call a cyclic submodule $N\subset \Omega^2$ (resp. $N\subset \Omega_n^2$) maximal if it is not contained in $T\Omega^2$ (resp. $T\Omega_n^2$). 
\end{definition}
\begin{remark}
     As $M_1$ and $M_2$ are isomorphic to $\Omega^\vee$, the modules $N_1$ and $N_2$ are maximal submodules of $\Omega^2$.  
\end{remark}
\begin{definition}
    Let $\mathcal{M}_n$ be the space of pairs of maximal submodules $(\bar{N}_1, \bar{N}_2)\in \Omega_n^2\times \Omega_n^2$. We define $\mathbb{P}_n$ to be the uniform distribution on $\mathcal{M}_n$. Let $\mathcal{A}_n$ be the power set of $\mathcal{M}_n$. Let $\mathcal{M}$ be the set of pairs $(N_1, N_2)$ of maximal submodules $N_i\in \Omega^2$. Let $\pi_n\colon \mathcal{M}\to \mathcal{M}_{n}$ and $\pi_{m,n}\colon \mathcal{M}_m\to \mathcal{M}_n$ for $m\ge n$ be the natural projections. Then we have $\mathcal{M}=\varprojlim_n \mathcal{M}_n$ and $((\mathcal{M}_n,\mathcal{A}_,\mathbb{P}_n),\pi_{m,n})$ forms a projective system of measurable spaces (\cite[Definition 2.2]{pinter}). Let $(\mathcal{M},\mathcal{A},\mathbb{P})$ be the inverse limit of this system (see \cite[Theorem 3.2]{pinter} for the existence).    
\end{definition}
\begin{remark}
   For any subset $X\subset \mathcal{M}$ we define $\mathbb{P}^*(X)=\inf_{X\subset B\in \mathcal{A}}\mathbb{P}(B)$. If $X\in \mathcal{A}$, then $\mathbb{P}(X)=\mathbb{P}^*(X)$. Note that $\Prob^*$ is an outer measure but not a measure.
\end{remark}

The aim of this section is to prove that $\{(N_1,N_2)\mid N_1\cap N_2=0\}$ is a measurable set and that $\Prob((N_1,N_2)\mid N_1\cap N_2=0)=1$. Applying this result to the modules $N_1$ and $N_2$ above, we obtain that $N_1+N_2=\Omega^2$, i.e. $(M_1\cap M_2)^\vee=\Omega^2/(N_1+N_2)=0$. In particular, $V_1(E)\cap V_2(E)$ is finite by Proposition \ref{prop}.

\begin{lemma}
\label{number of maximal submodules}
    With respect to notation above, there are $p^{n-1}(p+1)$ maximal submodules in $\Omega_n^2$.
\end{lemma}
\begin{proof}
    We prove the claim by induction on $n$ starting with $n=1$. In this case, we are counting the $\mathbb{F}_p$-lines of $\mathbb{F}_p^2$. We  have $p^2-1$ elements that generate these lines and $p-1$ units in $\mathbb{F}_p$. Thus, we obtain $(p+1)$ maximal submodules when $n=1$. Assume now that $n\geq 1$ and by induction, that we have already proven the following assertions.
    \begin{itemize}
        \item There are $p^{2n}-p^{2(n-1)}$ elements in $\Omega_n^2$ generating a maximal submodule. 
        \item There are $(p+1)p^{n-1}$ maximal submodules in $\Omega^2_n$.
    \end{itemize}
    Each element in $\Omega_n^2$ has $p^2$ preimages in $\Omega_{n+1}^2$, and an element in $\Omega_{n+1}^2$ generates a maximal submodule if and only if its restriction to $\Omega_n^2$ generates one. Thus, there are $p^{2(n+1)}-p^{2n}$ elements generating a maximal submodule in $\Omega^2_{n+1}$. This implies that there are
    \[\frac{p^{2n}(p^2-1)}{\#\Omega_{n+1}^\times}=\frac{p^{2n}(p^2-1)}{(p-1)p^{n}}=p^n(p+1)\]
    maximal submodules in $\Omega_{n+1}^2$
\end{proof}
\begin{lemma}
    \label{max-cyclic}
    Let $N_1$ and $N_2$ be maximal submodules. Then one of the following is true
    \begin{itemize}
        \item $N_1\cap N_2=0$.
        \item $N_1=N_2$.
    \end{itemize}
\end{lemma}
\begin{proof}
    Assume that $N_1\cap N_2\neq 0$. As $\Omega$ is a principal ideal domain, there exist non-negative integers $a,b$ such that 
    \[T^aN_1=T^bN_2.\]
    As $\Omega^2$ does not contain any $T$-torsion, we can assume that $a=0$. Thus, $N_1=T^bN_2$. As $N_1$ is maximal, $N_1$ is not contained in $T\Omega^2$. thus, $b=0$ and we obtain indeed $N_1=N_2$. 
\end{proof}
\begin{lemma}
    The set $\{(N_1,N_2)\mid N_1\cap N_2=0\}$ is measurable.
\end{lemma}
\begin{proof}
    As the union of countably many measurable sets is measurable it suffices to show that 
    \[\{(N_1,N_2)\mid N_1\cap N_2=0\}=\cup_n \{\pi_n(N_1)\neq \pi_n(N_2)\}.\]
By Lemma \ref{max-cyclic}, the right hand side is contained in the left hand side. Assume that there is a pair $(N_1,N_2)$ such that $\pi_n(N_1)= \pi_n(N_2)$ for all $n$. Let $x\in N_1$. For each $i$ we can find a decomposition
\[x=n_i+T^iy_i,\quad n_i\in N_2.\]
Let $i<j$, we obtain
\[n_i-n_j=T^i(T^{j-i}y_j-y_i). \]
Thus $(n_i)_{i\in \N}$ is a Cauchy sequence. As $\Omega$ is compact there is a convergent subsequence $(n_{i_k})$. Let $n$ be the limit. We obtain
\[x=n_{i_k}+T^{i_k}y_{i_k}\rightarrow n\in N_2.\]
Therefore, $x\in N_2$. Changing the roles of $N_1$ and $N_2$ proves $N_1=N_2$. Thus, if $N_1\neq N_2$, there exists an $n$ such that $\pi_N(N_1)\neq \pi_n(N_2)$, which implies that 
\[\{(N_1,N_2)\mid N_1\cap N_2=0\}\subset \cup_n \{\pi_n(N_1)\neq \pi_n(N_2)\}.\]

\end{proof}

\begin{theorem}
\label{thm:probability}
    Let $N_1,N_2$ be cyclic submodules of $\Omega^2$. Then we have
    \[\Prob(\{(N_1, N_2)\in \mathcal{M}\mid N_1\cap N_2=0\})=1.\]
\end{theorem}
\begin{proof}
By Lemma \ref{max-cyclic} we obtain 
    \begin{align*}
        \Prob(N_1\cap N_2=0)=\Prob(N_1\neq N_2)\\
        \ge \Prob(\pi_{n}(N_1)\neq \pi_n(N_2)).\\
    \end{align*}
    Note that the condition $\pi_n(N_1)\neq \pi_n(N_2)$ actually produces a measurable set. Therefore, we find that 
    \[\begin{split} & \Prob(\pi_{n}(N_1)\neq \pi_n(N_2)) \\
    = & \Prob_n(\pi_{n}(N_1)\neq \pi_n(N_2)) \\
    = &1-\Prob_n(\pi_{n}(N_1)= \pi_n(N_2)).\end{split}\]
    Let $\mathcal{N}_n$ be the set of maximal submodules in $\Omega^2$. 
     By Lemma \ref{number of maximal submodules} we have
    \begin{align*}\Prob_n(\pi_{n}(N_1)=\pi_n(N_2))=\sum_{M\in \mathcal{N}_n}\Prob_n(\pi_n(N_1)=M)\Prob(\pi_n(N_2)=M)\\=\sum_{M\in \mathcal{N}_n}\left(\frac{1}{(p+1)p^{n-1}}\right)^2=\frac{1}{p^{n-1}(p+1)} .\end{align*}
    This implies that 
    \[\Prob(N_1\cap N_2=0)\ge 1-\frac{1}{(p+1)p^{n-1}}\]
    for all $n$. The left hand side of the above inequality does not depend on $n$, thus, letting $n$ tend to infinity gives the claim.

\end{proof}

We are now ready to give a proof of Theorem \ref{thm A}.
\begin{proof}[Proof of Theorem A]
    Let $N_1$ and $N_2$ be maximal cyclic submodules of $\Omega^2$. Then according to the Lemma \ref{max-cyclic}, $N_1\cap N_2\neq 0$ if and only if $N_1=N_2$. Thus in particular, this implies that if $N_1 \cap N_2$ is finite, then $N_1\cap N_2=0$. It then follows from Theorem \ref{thm:probability} that \[\mathbb{P}\left(\{(N_1, N_2)\in \mathcal{M}\mid N_1\cap N_2\text{ is infinite}\}\right)=0.\]
    Heuristic \ref{heuristic main} then implies that  \[\lim_{x\rightarrow\infty} \frac{\# \mathcal{T}_p(x)}{\# \mathcal{C}(x)}=0.\] Since we assume Conjecture A of Coates and Sujatha, it follows as a consequence of part (2) of Proposition \ref{Propn on M1 and M2} that if the $\mu$-invariant of $\op{Sel}_{p^\infty}(E/\Q_{\op{cyc}})$ is non-zero, then $V_1(E)\cap V_2(E)$ is infinite. Therefore, we have that $\mathcal{S}_p(x)\subseteq \mathcal{T}_p(x)$ and we deduce that  \[\lim_{x\rightarrow\infty} \frac{\# \mathcal{S}_p(x)}{\# \mathcal{C}(x)}=0,\] proving the result. 
\end{proof}

\bibliographystyle{alpha}
\bibliography{references}
\end{document}